\documentclass[a4paper,11pt]{article}
\usepackage{amsmath, amsfonts, amscd, amssymb, amsthm, enumerate, xypic}

\def\Mat{\text{M}}
\def\GL{\text{GL}}

\def\id{\text{id}}

\newcommand{\Ker}{\operatorname{Ker}}
\newcommand{\Vect}{\operatorname{span}}
\newcommand{\im}{\operatorname{Im}}

\newcommand{\rk}{\operatorname{rk}}

\renewcommand{\setminus}{\smallsetminus}

\def\K{\mathbb{K}}
\def\R{\mathbb{R}}

\def\Q{\mathbb{Q}}
\def\N{\mathbb{N}}

\def\calG{\mathcal{G}}
\def\calH{\mathcal{H}}

\def\calL{\mathcal{L}}
\def\calM{\mathcal{M}}

\def\lcro{\mathopen{[\![}}
\def\rcro{\mathclose{]\!]}}

\theoremstyle{definition}
\newtheorem{Def}{Definition}

\theoremstyle{plain}
\newtheorem{theo}{Theorem}
\newtheorem{prop}[theo]{Proposition}

\newtheorem{lemme}[theo]{Lemma}

\theoremstyle{plain}

\theoremstyle{remark}
\newtheorem{Rems}{Remarks}

\newtheorem{ex}[Rems]{Example}

\title{The singular linear preservers of non-singular matrices
\footnote{E-mail address: dsp.prof@gmail.com}}
\author{Cl\'ement de Seguins Pazzis \\
\begin{footnotesize}
\emph{Lyc\'ee Priv\'e Sainte-Genevi\`eve, 2, rue
de l'\'Ecole des Postes, 78029 Versailles Cedex, FRANCE.}
\end{footnotesize}}

\begin{document}

\thispagestyle{plain}
\maketitle

\begin{abstract}
Given an arbitrary field $\K$, we reduce the determination of the singular endomorphisms $f$ of $\Mat_n(\K)$
such that $f(\GL_n(\K)) \subset \GL_n(\K)$ to the classification of $n$-dimensional division algebras over $\K$.
Our method, which is based upon Dieudonn\'e's theorem on singular subspaces of $\Mat_n(\K)$, also yields a proof for the classical non-singular case.
\end{abstract}



\section{Introduction}

Here, $\K$ will denote an arbitrary field and $n$ a positive integer.
We let $\Mat_{n,p}(\K)$ denote the set of matrices with $n$ rows, $p$ columns and entries in $\K$, and
$\GL_n(\K)$ the set of non-singular matrices in the algebra $\Mat_n(\K)$ of square matrices of order $n$.
The columns of a matrix $M \in \Mat_n(\K)$ will be written $C_1(M),C_2(M),\dots,C_n(M)$, so that
$$M=\begin{bmatrix}
C_1(M) & C_2(M) & \cdots & C_n(M)
\end{bmatrix}.$$
Given a vector space $V$, we let $\calL(V)$ denote the algebra of endomorphisms of $V$. \\
For non-singular $P$ and $Q$ in $\GL_n(\K)$, we define
$$u_{P,Q} : \begin{cases}
\Mat_n(\K) & \longrightarrow \Mat_n(\K) \\
M & \longmapsto P\,M\,Q
\end{cases}\quad \text{and} \quad v_{P,Q} :
\begin{cases}
\Mat_n(\K) & \longrightarrow \Mat_n(\K) \\
M & \longmapsto P\,M^t\,Q.
\end{cases}$$
Clearly, these are non-singular endomorphisms of the vector space $\Mat_n(\K)$
which map $\GL_n(\K)$ onto itself, and the subset
$$\calG_n(\K):=\bigl\{u_{P,Q} \mid (P,Q)\in \GL_n(\K)^2\bigr\} \cup \bigl\{v_{P,Q} \mid (P,Q)\in \GL_n(\K)^2\bigr\}$$
is clearly a subgroup of $\GL(\Mat_n(\K))$, which we will call the \textbf{Frobenius group}.
\vskip 3mm
\noindent Determining the endomorphisms of the vector space $\Mat_n(\K)$ which
preserve non-singularity has historically been one of the first successfull linear preserver problem, dating back to
Frobenius (\cite{Frobenius}), who successfully classified the linear preservers of the determinant,
and Dieudonn\'e (\cite{Dieudonne}), who classified the non-singular linear preservers of the general linear group.
Some improvements have been made later on the issue (cf. \cite{Marcus} and \cite{Botta}).
The following theorem is now folklore and essentially sums up what was known to this date:

\begin{theo}\label{classical}
\begin{enumerate}[(i)]
\item The group $\calG_n(\K)$ consists of all the endomorphisms $f$ of $\Mat_n(\K)$ such that $f(\GL_n(\K))=\GL_n(\K)$.
\item The group $\calG_n(\K)$ consists of all the endomorphisms $f$ of $\Mat_n(\K)$ such that $f^{-1}(\GL_n(\K))=\GL_n(\K)$.
\item The group $\calG_n(\K)$ consists of all the non-singular endomorphisms $f$ of $\Mat_n(\K)$ such that $f(\GL_n(\K)) \subset \GL_n(\K)$.
\item If $\K$ is algebraically closed, then $\calG_n$ consists of all the endomorphisms $f$ of $\Mat_n(\K)$ such that
$f(\GL_n(\K)) \subset \GL_n(\K)$.
\end{enumerate}
\end{theo}

\noindent Our main interest here is finding all the endomorphisms $f$ of $\Mat_n(\K)$
which stabilize $\GL_n(\K)$, i.e. $f(\GL_n(\K)) \subset \GL_n(\K)$.
The issue here is the existence of non-singular ones. Here are a few examples:

\begin{ex}
In $\Mat_2(\R)$, the endomorphism
$$\begin{bmatrix}
a & c \\
b & d
\end{bmatrix} \mapsto \begin{bmatrix}
a & -b \\
b & a
\end{bmatrix}$$
is singular and stabilizes $\GL_2(\R)$. Indeed, if $\begin{bmatrix}
a & c \\
b & d
\end{bmatrix} \in \GL_2(\R)$, then $(a,b)\neq (0,0)$ hence
$\begin{vmatrix}
a & -b \\
b & a
\end{vmatrix}=a^2+b^2>0$.
\end{ex}

\begin{ex}
In $\Mat_3(\Q)$, consider the companion matrix
$$A=\begin{bmatrix}
0 & 0 & 2 \\
1 & 0 & 0 \\
0 & 1 & 0
\end{bmatrix}.$$
Since the minimal polynomial $X^3-2$ of $A$ is irreducible over $\Q$, the subalgebra $\Q[A]$ is a field.
The singular endomorphism
$$M \longmapsto m_{1,1}.I_3+m_{2,1}.A+m_{3,1}.A^2$$
then clearly maps $\GL_3(\Q)$ into $\Q[A] \setminus \{0\}$ hence stabilizes $\GL_3(\Q)$.
\end{ex}

\noindent All those examples can be described in a normalized way.
We will need a few definitions first.

\begin{Def}
A linear subspace $V$ of $\Mat_n(\K)$ will be called \textbf{non-singular}  when
$V \setminus \{0\} \subset \GL_n(\K)$, and \textbf{full non-singular} when
in addition $\dim V=n$.
\end{Def}

\noindent Let $V$ be a full non-singular subspace of $\Mat_n(\K)$, with $n \geq 2$.
The projection onto the first column
$$\pi : \begin{cases}
V & \longrightarrow \Mat_{n,1}(\K) \\
M & \longmapsto C_1(M)
\end{cases}$$
is then a linear isomorphism. It follows that
$$\psi : \begin{cases}
\Mat_n(\K) & \longrightarrow \Mat_n(\K) \\
M & \longmapsto \pi^{-1}\bigl(C_1(M))\bigr.
\end{cases}$$
is a singular linear map which maps every non-singular matrix to a non-singular matrix.
More generally, given a non-zero vector $X \in \K^n$
and an isomorphism $\alpha : \K^n \overset{\simeq}{\rightarrow} V$,
the linear maps $M \mapsto \alpha(MX)$ and $M \mapsto \alpha(M^t\,X)$
are singular endomorphisms of $\Mat_n(\K)$ that stabilize $\GL_n(\K)$.

\vskip 2mm
\noindent
In this article, we will prove that the aforementioned maps are the only
singular preservers of $\GL_n(\K)$:

\begin{theo}[Main theorem]\label{princtheo}
Let $n \geq 2$.
Let $f$ be a linear endomorphism of $\Mat_n(\K)$ such that $f(\GL_n(\K)) \subset \GL_n(\K)$.
Then:
\begin{enumerate}[(i)]
\item either $f$ is bijective and then $f \in \calG_n(\K)$;
\item or there exists a full non-singular subspace $V$ of $\Mat_n(\K)$, an isomorphism
$\alpha : \K^n \overset{\simeq}{\rightarrow} V$ and a column $X \in \K^n \setminus \{0\}$ such that:
$$\forall M \in \Mat_n(\K), \; f(M)=\alpha(MX) \quad \text{or} \quad \forall M \in \Mat_n(\K), \; f(M)=\alpha(M^t\,X).$$
\end{enumerate}
As a consequence, if $f$ is singular, then $\im f$ is a full non-singular subspace of $\Mat_n(\K)$.
\end{theo}

The rest of the paper is laid out as follows:
\begin{itemize}
\item we will first easily derive theorem \ref{classical} from theorem \ref{princtheo};
\item afterwards, we will prove theorem \ref{princtheo} by using a theorem of Dieudonn\'e on
the singular subspaces of $\Mat_n(\K)$;
\item in the last section, we will explain how the existence of full non-singular subspaces of $\Mat_n(\K)$
is linked to the existence of $n$-dimension division algebras over $\K$. This will prove fruitful in the case $\K=\R$.
\end{itemize}

\section{Some consequences of the main theorem}

Let us assume theorem \ref{princtheo} holds, and use it to prove the various statements in theorem \ref{classical}.
The case $n=1$ is trivial so we assume $n \geq 2$.
Remark first that every $f \in \calG_n(\K)$ is an automorphism of $\Mat_n(\K)$
and satisfies all the conditions $f(\GL_n(\K)) \subset \GL_n(\K)$, $f(\GL_n(\K))=\GL_n(\K)$ and
$f^{-1}(\GL_n(\K))=\GL_n(\K)$.

\vskip 3mm
\noindent Statement (iii) is straightforward by theorem \ref{princtheo}.

\begin{proof}[Proof of statement (i)]
Let $f : \Mat_n(\K) \rightarrow \Mat_n(\K)$ be a linear map such that $f(\GL_n(\K))=\GL_n(\K)$.
By the next lemma, $\GL_n(\K)$ generates the vector space $\Mat_n(\K)$, so $f$ must be onto,
hence non-singular, and statement (iii) then shows that $f \in \calG_n(\K)$.
\end{proof}

\begin{lemme}
The vector space $\Mat_n(\K)$ is generated by $\GL_n(\K)$.
\end{lemme}

\begin{proof}
The result is obvious when $n=1$. We now assume $n \geq 2$.
Set $(E_{i,j})_{1 \leq i,j\leq n}$ the canonical basis of $\Mat_n(\K)$.
Then $E_{i,j}=(I_n+E_{i,j})-I_n \in \Vect \GL_n(\K)$ for all $i \neq j$. \\
On the other hand, letting $i \in \lcro 1,n\rcro$ and choosing arbitrarily $j \in \lcro 1,n\rcro \setminus \{i\}$, we find that
$I_n+E_{i,j}+E_{j,i}-E_{i,i}$ is non-singular, therefore
$$E_{i,i}=I_n-(I_n+E_{i,j}+E_{j,i}-E_{i,i})+E_{i,j}+E_{j,i} \in \Vect \GL_n(\K).$$
This proves that $\Vect \GL_n(\K)=\Mat_n(\K)$.
\end{proof}

\begin{proof}[Proof of statement (ii)]
Let $f : \Mat_n(\K) \rightarrow \Mat_n(\K)$ be a linear map such that $f^{-1}(\GL_n(\K))=\GL_n(\K)$.
Assume that $f$ is not injective. Then there would be a non-zero matrix $A \in \Mat_n(\K)$ such that $f(A)=0$,
and it would follow that $A+P$ is non-singular for every non-singular $P$
(since then $f(A+P)=f(P) \in \GL_n(\K)$).
Then any matrix $B$ equivalent to $A$ would also verify this property, in particular
$B:=\begin{bmatrix}
I_r & 0 \\
0 & 0
\end{bmatrix}$, with $r:=\rk A>0$. However $B+(-I_n)$ is singular. This proves that $f$ is one-to-one, hence
non-singular, and since $f(\GL_n(\K)) \subset \GL_n(\K)$, statement (iii) shows that $f \in \calG_n(\K)$.
\end{proof}

\begin{proof}[Proof of statement (iv)]
Assume $\K$ is algebraically closed. Then every non-singular subspace of $\Mat_n(\K)$
has dimension at most $1$: indeed, given two non-singular $P$ and $Q$ in $\Mat_n(\K)$, the polynomial
$\det(P+x\,Q)=\det(Q)\,\det(PQ^{-1}+x.I_n)$ is non-constant and must then have a root in $\K$.
It follows from theorem \ref{princtheo} that every linear map $f : \Mat_n(\K) \rightarrow \Mat_n(\K)$
which stabilizes $\GL_n(\K)$ belongs to $\calG_n(\K)$.
\end{proof}

\section{Proof of the main theorem}

The basic idea is to use a theorem of Dieudonn\'e to study the
subspace $f^{-1}(V)$ when $V$ is a singular subspace of $\Mat_n(\K)$, i.e.
one that is disjoint from $\GL_n(\K)$. This is essentially the idea in the original proof of Dieudonn\'e
(\cite{Dieudonne}) but we will push it to the next level by not assuming that $f$ is one-to-one.

\subsection{A reduction principle}

Let $f : \Mat_n(\K) \rightarrow \Mat_n(\K)$ be a linear map which stabilizes $\GL_n(\K)$, and
let $(P,Q)\in \GL_n(\K)$.
Then any of the maps $u_{P,Q} \circ f$, $f \circ u_{P,Q}$ and $M \mapsto f(M)^t$
is linear and stabilizes $\GL_n(\K)$. Moreover, it is easily checked that if any one of them
is of one of the types listed in theorem \ref{princtheo}, then $f$ also is.
Our proof will make a great use of that remark.

\subsection{A review of Dieudonn\'e's theorem}

\begin{Def}
A linear subspace of a $\K$-algebra is called \textbf{singular} when it contains no invertible element.
\end{Def}

\noindent For example, given an $i \in \lcro 1,n\rcro$, the subset of matrices $\Mat_n(\K)$
which have null entries on the $i$-th column is an $(n^2-n)$-dimensional singular subspace.

\begin{Def}
Let $E$ be a finite-dimensional vector space, $H$ an hyperplane\footnote{Here, by a hyperplane (resp. a line), we mean 
a \emph{linear} subspace of codimension one (resp. of dimension one). When we will exceptionally have to deal with 
affine subspaces, we will always specify it.} 
of $E$ and $D$ a line of $E$. We define:
\begin{itemize}
\item $\calL_D(E)$ as the set of endomorphisms $u$ of $E$ such that $D \subset \Ker u$ ;
\item $\calL^H(E)$ as the set of endomorphisms $u$ of $E$ such that $\im u \subset H$.
\end{itemize}
Then $\calL_D(E)$ and $\calL^H(E)$ are both $(n^2-n)$-dimensional singular subspaces of $\calL(E)$.
The singular subspace $\calL_D(E)$ will be said to be of \textbf{kernel-type}, and the singular subspace
$\calL^H(E)$ of \textbf{image-type}.
\end{Def}

\noindent The following theorem of Dieudonn\'e (\cite{Dieudonne}), later generalized by Flanders (\cite{Flanders}) and
Meshulam (\cite{Meshulam}), will be used throughout our proof:

\begin{theo}[Dieudonn\'e's theorem]\label{singularsubspaces}
Let $E$ be an $n$-dimensional vector space over $\K$, and $V$ a singular subspace of $\calL(E)$.
Then:
\begin{enumerate}[(a)]
\item one has $\dim V \leq n^2-n$;
\item if $\dim V=n^2-n$, then we are in one of the mutually exclusive situations:
\begin{itemize}
\item there is one (and only one) hyperplane $H$ of $E$ such that $V=\calL^H(E)$;
\item there is one (and only one) line $D$ of $E$ such that $V=\calL_D(E)$.
\end{itemize}
\end{enumerate}
\end{theo}

\subsection{Inverse image of a singular subspace of kernel-type}

In what follows, the algebras $\Mat_n(\K)$ will be canonically identified with the algebra $\calL(\K^n)$
of endomorphisms of $E:=\K^n$. Let $f : \Mat_n(\K) \rightarrow \Mat_n(\K)$
be an endomorphism which stabilizes $\GL_n(\K)$.
Notice that, given a line $D$ of $E$ and a non-zero vector $X \in D$, the singular subspace $\calL_D(E)$
is actually the kernel of the linear map $M \mapsto MX$ on $\Mat_n(\K)$.

\begin{lemme}
Let $X\in \K^n \setminus \{0\}$ and set $D:=\Vect(X)$. Then:
\begin{itemize}
\item either there is an hyperplane $H$ of $E$ such that
$f^{-1}(\calL_D(E))=\calL^H(E)$;
\item or there is a line $D'$ of $E$ such that
$f^{-1}(\calL_D(E))=\calL_{D'}(E)$.
\end{itemize}
Moreover, the linear map $M \mapsto f(M)X$ from $\Mat_n(\K)$ to $\K^n$ is onto.
\end{lemme}

\begin{proof}
Since the subspace $\calL_D(E)$ contains no non-singular matrix,
the assumption on $f$ garantees that $f^{-1}(\calL_D(E))$ is a singular subspace of $\Mat_n(\K)$.
Since $f^{-1}(\calL_D(E))$ is the kernel of $\alpha : M \mapsto f(M)X$, the rank theorem shows that
$\dim f^{-1}(\calL_D(E)) \geq n^2-n$. Theorem \ref{singularsubspaces} then shows our first statement,
hence another use of the rank theorem proves that $\dim f^{-1}(\calL_D(E))=n^2-n$ and $\alpha$ is onto.
\end{proof}

\noindent We will now show that the type of $f^{-1}\bigl(\calL_D(E)\bigr)$ (kernel or image) is actually independent
of the given line $D$. This will prove a lot harder than in Dieudonn\'e's original proof (\cite{Dieudonne}) because
$f$ is not assumed one-to-one.

\begin{prop}
Let $D_1$ and $D_2$ denote two distinct lines in $\K^n$.
Then the singular subspaces $f^{-1}(\calL_{D_1}(E))$ and $f^{-1}(\calL_{D_2}(E))$
are either both of kernel-type or both of image-type.
\end{prop}

\begin{proof}
We will use a \emph{reductio ad absurdum} by assuming there is a line $D$ and an hyperplane $H$ of $E$ such that
$f^{-1}(\calL_{D_1}(E))=\calL_D(E)$ and $f^{-1}(\calL_{D_2}(E))=\calL^H(E)$.
By right-composing $f$ with $u_{P,Q}$ for some well-chosen non-singular $P$ and $Q$,
and then left-composing $u_{I_n,R}$ for some well-chosen non-singular $R$,
we are reduced to the case $D_1=D=\Vect(e_1)$, $D_2=\Vect(e_2)$ and $H=\Vect(e_2,\dots,e_n)$,
where $(e_1,\dots,e_n)$ denotes the canonical basis of $\K^n$.
Then $f$ has the following properties:
\begin{itemize}
\item Any matrix with first column $0$ is mapped by $f$ to a matrix
with first column $0$, and $M \mapsto C_1(f(M))$ is onto.
\item Any matrix with first line $0$ is mapped by $f$ to a matrix
with second column $0$, and $M \mapsto C_2(f(M))$ is onto.
\end{itemize}
By the factorization theorem for linear maps (\cite{Greub} proposition I p.45), we deduce that there are two isomorphisms
$\alpha : \Mat_{n,1}(\K) \overset{\simeq}{\longrightarrow} \Mat_{n,1}(\K)$ and
$\beta : \Mat_{1,n}(\K) \overset{\simeq}{\longrightarrow} \Mat_{n,1}(\K)$
such that, for every
$$M=\begin{bmatrix}
C & \cdots
\end{bmatrix}=\begin{bmatrix}
L \\
\vdots
\end{bmatrix} \quad \text{with $C \in \Mat_{n,1}(\K)$ and $L \in \Mat_{1,n}(\K)$,}$$
one has
$$f(M)=\begin{bmatrix}
\alpha(C) & \beta(L) & \cdots
\end{bmatrix}.$$
Set now $C_1:=\alpha \begin{bmatrix}
1 \\
0 \\
\vdots \\
0
\end{bmatrix}$ and $C_2:=\beta \begin{bmatrix}
1 & 0 & \cdots & 0
\end{bmatrix}$. We then recover two injective linear maps
$\alpha' : \Mat_{n-1,1}(\K) \hookrightarrow \Mat_{n,1}(\K)$ and
$\beta' : \Mat_{1,n-1}(\K) \hookrightarrow \Mat_{n,1}(\K)$ such that for every
$M=\begin{bmatrix}
1 & L \\
C & ?
\end{bmatrix} \in \Mat_n(\K)$ with first coefficient $1$, one has
$$f(M)=\begin{bmatrix}
C_1+\alpha'(C) & C_2+\beta'(L) & ?
\end{bmatrix}.$$
Let $(L,C)\in \Mat_{1,n-1}(\K) \times \Mat_{n-1,1}(\K)$.
Notice then that there exists an $N \in \Mat_{n-1}(\K)$ such that
$M=\begin{bmatrix}
1 & L \\
C & N
\end{bmatrix}$ is non-singular. Indeed, the matrix $N:=CL+I_{n-1}$ fits this condition
(remark that
$\begin{bmatrix}
1 & L \\
C & CL+I_{n-1}
\end{bmatrix}=\begin{bmatrix}
1 & 0 \\
C & I_{n-1}
\end{bmatrix}\begin{bmatrix}
1 & L \\
0 & I_{n-1}
\end{bmatrix}$).
For any such $M$, the matrix $f(M)$ must then be non-singular,
which proves that $C_1+\alpha'(C)$ and $C_2+\beta'(L)$ are linearly independent. \\
However, this has to hold for every pair $(L,C)\in \Mat_{1,n-1}(\K) \times \Mat_{n-1,1}(\K)$.
Therefore no vector in the affine hyperplane $\calH_1:=C_1+\im \alpha'$
is colinear to a vector in the affine hyperplane $\calH_2:=C_2+\im \beta'$.
There finally lies a contradiction: indeed, should we choose a vector $x_0$ in $E \setminus (\im \alpha' \cup \im \beta')$
(classically, such a vector exists because $E$ is never the union of two strict linear subspaces),
then the line $\Vect(x_0)$ would have to intersect both hyperplanes $\calH_1$ and $\calH_2$.
\end{proof}

\noindent We may actually assume there is some line $D$ such that $f^{-1}(\calL_D(E))$
has kernel-type, because, if not, we may replace $f$ with $M \mapsto f(M^t)$.
Therefore we may now assume, without loss of generality:
\begin{center}
For every line $D$ of $E$, there is a line $D'$ of $E$ such that $f^{-1}(\calL_D(E))=\calL_{D'}(E)$.
\end{center}

\subsection{Reducing the problem further}\label{reduc}

We let here $(e_1,\dots,e_n)$ denote the canonical basis of $E=\K^n$
and set $D_i:=\Vect(e_i)$ for every $i \in \lcro 1,n\rcro$.
We now have $n$ lines $D'_1,\dots,D'_n$ in $E$ such that
$\forall i \in \lcro 1,n\rcro, \; f^{-1}(\calL_{D_i}(E))=\calL_{D'_i}(E)$.
In every line $D'_i$, we choose a non-zero vector $x_i$. \\
Set $F:=\Vect(x_1,\dots,x_n)$ and $p:=\dim F$.
From $(x_1,\dots,x_n)$ can be extracted a basis of $F$.
\begin{itemize}
\item Replacing $f$ with $M \mapsto f(M)P$ for some suitable
permutation matrix $P$, we may assume $(x_1,\dots,x_p)$ is a basis of $F$.
\item Replacing $f$ with $M \mapsto f(MP)$ for some non-singular $P \in \GL_n(\K)$,
we may finally assume $(x_1,\dots,x_p)=(e_1,\dots,e_p)$, so that $F=\Vect(e_1,\dots,e_p)$.
\end{itemize}
After these reductions, let us restate some of the assumptions on $f$:
for every $i \in \lcro 1,p\rcro$ and every $M \in \Mat_n(\K)$,
if the $i$-th column of $M$ is $0$, then the $i$-th column of $f(M)$ is also $0$,
and $N \mapsto C_i(f(N))$ is onto (from $\Mat_n(\K)$ to $\Mat_{n,1}(\K)$).
By the factorization theorem for linear maps, we recover $p$ automorphisms
$\alpha_1,\dots,\alpha_p$ of $\Mat_{n,1}(\K)$ such that, for every
$M=\begin{bmatrix}
C_1 & C_2 & \cdots & C_p & ?
\end{bmatrix}$ in $\Mat_n(\K)$, one has:
$$f(M)=\begin{bmatrix}
\alpha_1(C_1) & \alpha_2(C_2) & \cdots & \alpha_p(C_p) & ?
\end{bmatrix}.$$

\vskip 2mm
\noindent We will now reduce the previous situation to the case $\alpha_1=\alpha_2=\cdots=\alpha_p=\id$.

\begin{lemme}\label{libre}
Under the previous assumptions, let $(C_1,\dots,C_p)\in \Mat_{n,1}(\K)^p$ be a linearly independent $p$-tuple.
Then $\bigl(\alpha_1(C_1),\dots,\alpha_p(C_p)\bigr)$ is linearly independent.
\end{lemme}

\begin{proof}
Indeed, $(C_1,\dots,C_p)$ can be extended into a basis
$(C_1,\dots,C_n)$ of $\Mat_{n,1}(\K)$.
Since $M:=\begin{bmatrix}
C_1 & \cdots & C_n
\end{bmatrix}$ is non-singular, $f(M)$ also is, which proves our claim.
\end{proof}

\vskip 2mm
\noindent
Define then $P \in \GL_n(\K)$ as the matrix canonically associated to $\alpha_1$. Then
we may replace $f$ with $f \circ u_{P^{-1},I_n}$, which changes no previous assumption.
In this case, $\alpha_1=\id_{\Mat_{n,1}(\K)}$.
We claim then that $\alpha_2,\dots,\alpha_p$ are scalar multiples of the identity.
Consider $\alpha_2$ for exemple. Since any linearly independent pair $(C_1,C_2)$ in $\Mat_{n,1}(\K)$
can be extended into a linearly independent $p$-tuple in $\Mat_n(\K)$, lemma \ref{libre} shows $(C_1,\alpha_2(C_2))$
must be linearly independent. It follows that for every $C \in \Mat_{n,1}(\K)$, the matrices $C$ and $(\alpha_2)^{-1}(C)$
must be linearly dependent. Classically, this proves $(\alpha_2)^{-1}$ is a scalar multiple of $\id$,
hence $\alpha_2$ also is. The same line of reasoning also shows that
this is true of $\alpha_3,\dots,\alpha_p$.

\vskip 2mm
\noindent We thus find non-zero scalars $\lambda_2,\dots,\lambda_p$ such that, for every
$M=\begin{bmatrix}
C_1 & C_2 & \cdots & C_p & ?
\end{bmatrix}$ in $\Mat_n(\K)$, one has
$f(M)=\begin{bmatrix}
C_1 & \lambda_2.C_2 & \cdots & \lambda_p.C_p & ?
\end{bmatrix}$.

\vskip 2mm
\noindent By replacing $f$ with $f \circ u_{I_n,P^{-1}}$ for
$P:=D(1,\lambda_2,\dots,\lambda_p,1,\dots,1)$, we are thus reduced to the following situation: \\
For every $M=\begin{bmatrix}
C_1 & C_2 & \cdots & C_p & ?
\end{bmatrix}$ in $\Mat_n(\K)$, one has
$f(M)=\begin{bmatrix}
C_1 & C_2 & \cdots & C_p & ?
\end{bmatrix}$.

\subsection{The coup de gr\^ace}

\begin{itemize}
\item If $p=n$, then we are reduced to the case $f=\id_{\calM_n(\K)}$,
in which $f=u_{I_n,I_n}.$
\item Assume $p=1$. \\
Then $\Ker f$ is the set of matrices with $0$ as first column.
Indeed, since $\underset{k=1}{\overset{n}{\bigcap}}\calL_{D_k}(E)=\{0\}$, we find
$$\Ker f=\underset{k=1}{\overset{n}{\bigcap}}f^{-1}\bigl(\calL_{D_k}(E)\bigr)
=\underset{k=1}{\overset{n}{\bigcap}}\calL_{D'_k}(E)=\calL_{D_1}(E).$$
By the factorization theorem for linear maps, we find a linear injection $g : \K^n \hookrightarrow \Mat_n(\K)$
such that $\forall M \in \calM_n(\K), \; f(M)=g(Me_1)$, where $e_1=\begin{bmatrix}
1 \\
0 \\
\vdots \\
0
\end{bmatrix}$.
Notice then that $\im g=\im f$ and $\im g$ is an $n$-dimensional linear subspace of $\Mat_n(\K)$. \\
Finally, $\im g$ is actually non-singular: indeed, for every $x \in \K^n \setminus \{0\}$, there exists $M \in \GL_n(\K)$ such that
$Me_1=x$, hence $g(x)=f(M)$ is non-singular. We have thus proven that $f$ verifies condition (ii) in theorem \ref{princtheo}.
\end{itemize}
\vskip 2mm
Our proof of theorem \ref{princtheo} will then be finished should we prove that
only the above two cases can arise. Assume then $1<p<n$ and consider the vector
$x_{p+1}$. Notice that we now simply have $f^{-1}(\calL_D(E))=\calL_D(E)$ for any line $D$ of $F=\Vect(e_1,\dots,e_p)$.
Moreover, the situation is left unchanged should we choose a non-singular $P \in \GL_p(\K)$, set
$Q:=\begin{bmatrix}
P & 0 \\
0 & I_{n-p}
\end{bmatrix}$ and replace $f$ with $u_{I_n,P^{-1}} \circ f \circ u_{I_n,P}$.
It follows that we may actually assume $D'_{p+1}=D_1$ in addition to the previous assumptions
(at this point, the reader must check that none of the previous reductions changes the lines $D_{p+1},\dots,D_n$).

\vskip 2mm
\noindent
Another use of the factorization theorem then helps us find an endomorphism $\alpha$ of $\Mat_{n,1}(\K)$
such that, for every $M=\begin{bmatrix}
C_1 & C_2 & \cdots & C_p & ?
\end{bmatrix}$ in $\calM_n(\K)$, one has
$f(M)=\begin{bmatrix}
C_1 & C_2 & \cdots & C_p & \alpha(C_1) & ?
\end{bmatrix}$.
Borrowing an argument from section \ref{reduc}, we deduce that for any linearly independent pair $(C_1,C_2)$ in $\Mat_{n,1}(\K)$, the triple
$(C_1,C_2,\alpha(C_1))$ is also linearly independent (this is where the assumption $1<p<n$ comes into play).
Clearly, this is absurd: indeed, choose $C_1$ arbitrarily in $\Mat_{n,1}(\K) \setminus \{0\}$, then
$C_2:=\alpha(C_1)$ if $(C_1,\alpha(C_1))$ is linearly independent, and choose arbitrarily
$C_2$ in $\Mat_{n,1}(\K) \setminus \Vect(C_1)$ if not (there again, we use $p \geq 2$).
This contradiction shows $p \in \{1,n\}$, which completes our proof of theorem \ref{princtheo}.

\section{A link with division algebras}

We will show here how the full non-singular subspaces of $\Mat_n(\K)$ are connected to division algebra over $\K$.
Let us recall first a few basic facts about them.

\begin{Def}
A \textbf{division algebra} over $\K$ is a $\K$-vector space $D$ equipped with a bilinear map
$\star : D \times D \rightarrow D$ such that $x \mapsto a\star x$ and $x \mapsto x \star a$ are automorphisms of $D$
for every $a \in D \setminus \{0\}$.
\end{Def}

\noindent Of course, every field extension of $\K$, and more generally every skew-field extension of $\K$ is a division algebra over $\K$.
There are however non-associative division algebras, the most famous example being the algebra of octonions (see
\cite{Conway} for an extensive treatment on them).

\begin{Rems}
\begin{enumerate}[(a)]
\item Note that associativity is not required on the part of $\star$ !
\item If $D$ is finite-dimensional, then the latter condition in the definition of a division algebra
is verified if and only if $x \mapsto a\star x$
is bijective for every $a \in D \setminus \{0\}$. The data of $\star$ is then equivalent to that of a linear map
$$\alpha : D \longrightarrow \calL(D)$$
which maps $D \setminus \{0\}$ into $\GL(D)$
(indeed, to such a map $\alpha$, we naturally associate the pairing $(a,b)\mapsto \alpha(a)[b]$).
\end{enumerate}
\end{Rems}

\vskip 3mm
\noindent The correspondence between full non-singular subspaces of $\GL_n(\K)$ and division algebras over $\K$
is now readily explained:
\begin{itemize}
\item Let $V$ be a full non-singular subspace $V$ of $\GL_n(\K)$. Setting a basis of $V$, we define an isomorphism
$\theta : \K^n \overset{\simeq}{\rightarrow} V$ which induces an isomorphism of algebras $\overline{\theta} : \Mat_n(\K) \overset{\simeq}{\rightarrow} \calL(V)$.
Restricting $\overline{\theta}$ to $V$ then gives rise to a division algebra structure on $V$.
\item Conversely, given a division algebra $D$ with structural map
$\alpha : D \rightarrow \calL(D)$, we can choose a basis of $D$, which defines an algebra isomorphism
$\psi : \calL(D) \overset{\simeq}{\rightarrow} \Mat_n(\K)$, and then associate to $D$ the full non-singular subspace
$\psi(\alpha(D))$ of $\Mat_n(\K)$.
\end{itemize}
Working with the canonical basis of $\K^n$, we have just established a bijective correspondence between
the set of structures of division algebras on $\K^n$ (which extend its canonical vector space structure), and
the set of full non-singular subspaces of $\Mat_n(\K)$.

\vskip 2mm
\noindent By combining our main theorem with the Bott-Milnor-Kervaire theorem on division algebras over the real numbers
(cf. \cite{BottMilnor} and \cite{Kervaire}), this yields:

\begin{prop}
Let $n \in \N \setminus \{2,4,8\}$. Then every
linear endomorphism $f$ of $\Mat_n(\R)$ which stabilizes $\GL_n(\R)$
belongs to the Frobenius group $\calG_n(\R)$.
\end{prop}

\section*{Acknowledgements}

I would like to thank Fran\c cois Lussier for insisting I should tackle the problem
considered here.

\end{document}